\documentclass[]{amsart}

\usepackage{subfigure}
\usepackage[justification=centering]{caption}
\usepackage[utf8]{inputenc}
\usepackage[toc]{appendix}

\usepackage{amssymb}
\usepackage{amsthm}
\usepackage{amsmath,a4wide}
\usepackage{amsbsy}
\usepackage[all]{xy}
\usepackage{bm}
\usepackage{hyperref}
\usepackage{amsthm}
\usepackage{marvosym}
\usepackage{verbatim} 
\usepackage{graphicx}
\usepackage{enumerate}
\usepackage{graphicx}
\usepackage{tikz,pgfplots}
\usetikzlibrary{shapes}
\usetikzlibrary{intersections}
\usepackage{pdfsync}
\usepackage{caption}
\usepackage{hyperref}
\usepackage{enumerate}

\theoremstyle{plain}
\newtheorem{theorem}{Theorem}[section]
\theoremstyle{plain}

\theoremstyle{plain}
\newtheorem{lemma}[theorem]{Lemma}
\newtheorem*{lem}{Lemma}
\theoremstyle{plain}

\theoremstyle{definition}

\theoremstyle{remark}

\theoremstyle{remark}

\theoremstyle{definition}

\theoremstyle{definition}

\theoremstyle{definition}

\providecommand{\norm}[1]{\lVert#1\rVert}

\newcommand{\R}{\mathbb{R}}

\newcommand{\Rd}{\mathbb{R}^d}
\newcommand{\Z}{\mathbb{Z}}
\newcommand{\N}{\mathbb{N}}
\newcommand{\Lt}[1][d]{L^2(\R^{#1})}
\newcommand{\G}{\mathcal{G}}
\newcommand{\F}{\mathcal{F}}

\renewcommand{\l}{\lambda}
\renewcommand{\L}{\Lambda}

\begin{document}

\title[Optimal Frame Bounds and Jacobi Theta Functions]{Optimal Gabor Frame Bounds for Separable Lattices and Estimates for Jacobi Theta Functions}

\author[M. Faulhuber]{Markus Faulhuber}
\address{NuHAG, Faculty of Mathematics, University of Vienna, Oskar-Morgenstern-Platz 1, 1090 Vienna, Austria}
\email{markus.faulhuber@univie.ac.at}

\author[S. Steinerberger]{Stefan Steinerberger}
\address{Department of Mathematics, Yale University, 10 Hillhouse Avenue, New Haven, CT 06510, USA}
\email{stefan.steinerberger@yale.edu}

\begin{abstract}
	We study sharp frame bounds of Gabor frames for integer redundancy with the standard Gaussian window and prove that the square lattice optimizes both the lower and the upper frame bound among all rectangular lattices. This proves a conjecture of Floch, Alard \& Berrou (as reformulated by Strohmer \& Beaver). The proof is based on refined log-convexity/concavity estimates for the Jacobi theta functions $\theta_3$ and $\theta_4$.
\end{abstract}


\maketitle

\section{Introduction}\label{sec_Intro}

\subsection{Introduction.} The study of Gabor frames originates in a 1946 paper of Gabor \cite{Gab46} in which he describes intermediate cases between pure time analysis and pure frequency analysis (Fourier analysis). He proposes to have a two-dimensional representation of a one-dimensional signal (function) which simultaneously uses information of the distribution of the signal in time and its frequencies. 
A Gabor system (or Weyl-Heisenberg system) for $\Lt$ is generated by a (fixed, non-zero) window function $g \in \Lt$ and an index set $\L \subset \R^{2d}$ and is denoted by $\G(g,\L)$. It consists of time-frequency shifted versions of $g$. We say $\l = (x, \omega) \in \Rd \times \Rd$ is a point in the time-frequency plane and use the following notation for a time-frequency shift by $\l$
$$
	\pi (\l)g(t) = M_\omega T_x \, g(t) = e^{2 \pi i \omega \cdot t} g(t-x), \quad x,\omega,t \in \Rd.
$$
Hence, for a window function $g$ and an index set $\L$ the Gabor system is
$$
	\G(g,\L) = \lbrace \pi(\l) g \, | \, \l \in \L \rbrace.
$$
The time-frequency shifted versions of the window $g$ are called atoms. In order to be a frame, $\G(g,\L)$ has to satisfy the frame inequality
$$
	A \norm{f}_2^2 \leq \sum_{\l \in \L} \left| \langle f, \pi(\l) g \rangle \right|^2 \leq B \norm{f}_2^2, \quad \forall f \in \Lt
$$
for some positive constants $0 < A \leq B < \infty$ called frame constants or frame bounds. Whenever we speak of frame bounds, we only consider the optimal frame bounds. The frame bounds serve as quantitative measurement of how close the frame is to a tight frame, in which case we would have $A=B$. If the frame gives rise to an orthonormal basis, we have $A=B=1$. The index set $\L \subset \R^{2d}$ is called a lattice if it is generated by an invertible (non-unique) $2d \times 2d$ matrix $S$, in the sense that $\L = S \Z^{2d}$. The volume of the lattice, which is unique, is defined as
$$
	\mbox{vol}(\L) = |\det(S)| \qquad \mbox{while its density or redundancy is given by} \qquad \delta(\L) = \frac{1}{\mbox{vol}(\L)}.
$$
A lattice is called separable if the generating matrix can take the form
$$
	S = \left(
	\begin{array}{c c}
		\alpha I & 0\\
		0 & \beta I
	\end{array}
	\right).
$$

For more details on frames, Gabor frames and time-frequency analysis we refer to the classical texts \cite{Chr03}, \cite{Fei81}, \cite{FeiGro97}, \cite{FeiLue15}, \cite{Gro01}, \cite{Hei06}. One of the fundamental questions in Gabor analysis is to understand when a Gabor system $\G(g,\L)$ forms a frame. For a fixed window $g$ the family of all lattices $\L$ which together with $g$ generate a frame is called the frame set of the window $g$. We distinguish between the full frame set whose elements are lattices in general and the reduced frame set whose elements are the lattice parameters of separable lattices \cite{Gro14}. For a window function  $g \in \Lt$ we denote the full frame set by
$$
		\F_{full}(g) = \lbrace \Lambda \subset \R^{2d} \, ~~\text{lattice } | ~ \, \G(g,\L) \text{ is a frame} \rbrace
$$
and the reduced frame set by
$$
	\F_{(\alpha, \beta)}(g) = \lbrace (\alpha,\beta) \in \R_+ \times \R_+ \, |~ \, \G(g,\alpha \Z^d \times \beta \Z^d) \text{ is a frame} \rbrace.
$$

Clearly, $(\alpha,\beta) \in \F_{(\alpha,\beta)}(g)$ implies $\alpha \Z^d \times \beta \Z^d \in \F_{full}(g)$. We may rephrase the question about when a Gabor system forms a frame in the following way. For any given $g$ what is its (full or reduced) frame set? At this point we want to emphasize that there is no general idea of how to determine the frame set of a class of functions or even a single function. Even less is known about how the frame bounds change within the frame set.
The 1-dimensional standard Gaussian window $g_0(t) = 2^{1/4} e^{-\pi t^2}$ has been fully analyzed: results of Lyubarskii \cite{Lyu92} and Seip \cite{Sei92} give the full frame set for Gabor frames with a Gaussian window $g$ as
$$
	\F_{full}(g) = \lbrace \Lambda \subset \R^2 \, | \, \text{vol}(\Lambda) < 1 \rbrace.
$$
This implies that the reduced frame set is given by
$$
	\F_{(\alpha,\beta)}(g) = \lbrace (\alpha, \beta) \in \R_+ \times \R_+ \, | \, \alpha \beta < 1 \rbrace.
$$
However, it is still not clear how to find the lattice that optimizes the frame bounds.

\section{Statement of results}
\subsection{Main result.} Given a Gaussian window function, which lattice $\Lambda \subset \mathbb{R}^2$ minimizes $B/A$?  Floch, Alard \& Berrou \cite{FloAlaBer95} conjectured in 1995 that the square lattice yields the optimal configuration. This was disproved in 2003 by Strohmer \& Beaver \cite{StrBea03} who conjecture that $B/A$ is minimized for the hexagonal lattice among all lattices of fixed redundancy (in 2012 Abreu \& D\"orfler \cite{AbrDoe12} also conjectured that this should be true for any redundancy greater than 1). Strohmer \& Beaver claim that it is `plausible' to assume that the square lattice optimizes the frame bounds among all \textit{rectangular} lattices. For integer redundancy we prove this to be the case by showing an even stronger result.

\begin{theorem}[Main result]\label{main}
	Consider the window function $g_0(t) = 2^{1/4} e^{-\pi t^2}$. Among all separable lattices with $(\alpha \beta)^{-1} \in \mathbb{N}$ fixed, the square lattice maximizes $A$ and minimizes $B$.
\end{theorem}
After a preliminary reduction (based on work by Janssen \cite{Jan96}), we need fine estimates on the Jacobi theta function restricted to a vertical strip in the upper half plane with $\Re z = 0$ and $\Re z = 1/2$, respectively. We introduce these functions explicitly as
$$
	\theta_3(s) = \sum_{k=-\infty}^{\infty} e^{-\pi k^2 s} \qquad \mbox{and} \qquad \theta_4(s) = \sum_{k=-\infty}^{\infty} (-1)^k e^{-\pi k^2 s}.
$$
They appear in many different areas:  $\theta_4$, for example, is a rescaling of the c.d.f.\ of the Kolmogorov-Smirnov distribution in probability theory. Theorem \ref{main} naturally splits into 4 separate statements
$$
	\left\{(\alpha \beta)^{-1}~\mbox{is even}, ~(\alpha \beta)^{-1} ~\mbox{is odd}\right\}  \times \left\{\mbox{maximization of $A$, minimization of $B$}\right\}
$$
and we will prove each statement separately. Tolimieri \& Orr \cite{TolOrr95} mention that one of the four cases, the square lattice minimizing the upper frame bound for even redundancy, has been shown by Janssen (unpublished). Somewhat to our surprise, the four proofs require four different ideas and are very different in style -- only one proof is straightforward. A common theme is that $\theta_3(s)$ and $\theta_4(s)$ are easy to understand for large values of $s$ but exhibit more complicated behavior for $s$ small (because many different terms contribute to the sum).

\subsection{Even redundancy.} We start with $(\alpha \beta)^{-1} \in 2 \mathbb{N}$. In that special case, the two desired statements follow from the following two inequalities, respectively.

\begin{theorem}\label{alternativ}
	For all $r,s > 0$,
	$$
		\theta_3(rs) \theta_3\left(\frac{r}{s}\right) \geq \theta_3(r)^2  \quad \mbox{and} \quad
		\theta_4(rs) \theta_4\left(\frac{r}{s}\right) \leq \theta_4(r)^2
	$$
	with equality only for $s=1$.
\end{theorem}
Our proofs of these inequalities are non-trivial: the first one requires a new identity for $\theta_3$ while the other one uses an algebraic monotonicity property arising when interpreting $\theta_4$ as an infinite product. Our approach uses the notions of log-concavity and log-convexity. It is very easy to show $\log{\theta_3(s)}$ is a convex function (see Lemma \ref{lem_theta3_log}): indeed, more generally, the Cauchy-Schwarz inequality immediately implies for $a_k, b_k \geq 0$ and
$$
	f(s) = \sum_{k=1}^{\infty}{a_k e^{- b_k s}} \qquad \mbox{that} \quad \left(\log{(f)}\right)'' \geq 0.
$$
We are not aware of a similarly easy way to establish log-concavity of $\theta_4(s)$ and could not find the result in the literature. Recall that log-concavity of a function $f(s)$ can be equivalently written as
$$
	f''(s)f(s) - f'(s)^2 \leq 0
$$
while log-convexity can be equivalently written as
$$
	f''(s)f(s) - f'(s)^2 \geq 0.
$$
We require quantitatively stronger results which will then imply Theorem \ref{alternativ}. There has been some recent work on refined estimates for $\theta-$functions (e.g.\ \cite{dixit, schief, soly}) and we believe that our inequalities could be of independent interest.

\begin{theorem}[Refined logarithmic convexity for $\theta_3$]\label{log_convex}
	We have, for $s > 0$,
$$
	\theta_3''(s)\theta_3(s) - \theta_3'(s)^2 > -\frac{\theta_3'(s) \theta_3(s)}{s} > 0.
$$
\end{theorem}
The proof is quite curious: $\theta_3(s)$ is complicated to evaluate if $s$ is close to the origin (because many different terms start to contribute to the sum) but is quite simple on $\left\{s \in \mathbb{R}: s \geq 1\right\}$ 
because it is essentially dominated by its leading term. We first establish the relevant inequality in $\left\{s \in \mathbb{R}: s \geq 1\right\}$  and then use the new identity (which follows 
from the Jacobi identity)
$$
	s \frac{\theta_3'(s)}{\theta_3(s)} + \frac{1}{s}\frac{\theta_3'\left(\frac{1}{s}\right)}{\theta_3\left(\frac{1}{s}\right)} = -\frac{1}{2} \qquad \quad \mbox{for}~s > 0
$$
to backpropagate the information to the origin.

\begin{theorem}[Refined logarithmic concavity for $\theta_4$]\label{log_concav}
	We have, for $s > 0$,
	$$
		\theta_4''(s)\theta_4(s) - \theta_4'(s)^2 < -\frac{\theta_4'(s) \theta_4(s)}{s} < 0.
	$$
\end{theorem}
This inequality is new (however, see Coffey \& Csordas \cite{csordas} for a related function). The proof is based on exploiting a suitable monotonicity property using an infinite product representation of $\theta_4$. Since $\theta_4$ has alternating signs, it is difficult to handle and the proof contains a certain 'magic' element of algebraic simplification. Interestingly, this argument does not work for $\theta_3$ even though there exists an analogous representation of $\theta_3$ as an infinite product. Theorem \ref{log_convex} implies that $s \theta_3'(s)/\theta_3(s)$ is monotonically increasing on $\mathbb{R}^{+}$ while Theorem \ref{log_concav} yields that $s \theta_4'(s)/\theta_4(s)$ is monotonically decreasing. Stronger results seem to be true: $s^2 \theta_4'(s)/\theta_4(s)$ seems to be monotonically decreasing and convex but our arguments cannot show that.

\subsection{Odd redundancy} We will now consider the case  $(\alpha \beta)^{-1}$ being an odd integer. Following again Janssen \cite{Jan96}, we introduce
the function
$$ \theta_{o}(s) = \sum_{k \in \mathbb{Z}}{e^{-\pi(2k+1)^2 s}},$$
which can be understood as $\theta_3$ with summation restricted to the odd integers. The upper frame bound will turn out to be implied by the following statement.
\begin{theorem}[Odd redundancy, upper frame bound]\label{oddupp}
	For $r,s \in \R_+$
	$$
		\theta_3(rs) \theta_3(r/s) - 2 \theta_{o}(r s) \theta_{o}(r/s) \qquad \mbox{is minimal for}~s=1.
	$$
\end{theorem}

We know from Theorem \ref{alternativ} that  $ \theta_3(rs) \theta_3(r/s) \geq \theta_3(r)^2$. There is a very helpful algebraic simplification: $ \theta_{o}(r s) \theta_{o}(r/s)$ is
maximal for $s=1$ and once this is proven, it will imply the result. By using Poisson summation the statement about the maximality of $\theta_o(rs) \theta_o(r/s)$ can be traced back to the problem for the $\theta_4$ function in Theorem \ref{alternativ}.

\begin{theorem}[Odd redundancy, lower frame bound]\label{oddlow}
	Let $r \geq 1$, $s \in \R_+$. Then
	$$
		\theta_4(rs) \theta_4\left(\frac{r}{s}\right) - 2 \theta_{o}(r s) \theta_{o}\left( \frac{r}{s} \right) \qquad \mbox{is maximal for}~s=1.
	$$
\end{theorem}

This proof is a straightforward combination of the already achieved results: by the previous analysis, both functions have a global maximum in $s=1$ and this is their only critical point. They are both strictly increasing on $0 < s < 1$ and strictly decreasing for $s > 1$. It is easy to show that the derivative of the first term is much larger than the derivative of the second term. Hence, the difference of the two products takes its maximum for $s=1$.

\subsection{Open problems.}
The problem under consideration concerns the so-called fine structure of Gabor frames, which refers to relations between properties of a fixed window and its corresponding frame set (as opposed to coarse structure referring to general properties of the frame set). Hardly anything is known about the fine structure and Gr\"ochenig even goes so far as to describe it as 'mysterious' \cite{Gro14}. Strohmer \& Beaver \cite{StrBea03} conjecture, based on numerical results, that for lattices of fixed redundancy and Gaussian window the smallest value $B/A$ is achieved for a hexagonal lattice and notice that the value is 'suspiciously close to $\sqrt[3]{2}$'. 

\begin{figure}[ht!]
	\begin{minipage}[l]{.3\textwidth}
		\begin{center}
			\begin{tikzpicture}
				\begin{scope}[%
					every node/.style={anchor=west,
					regular polygon, 
					regular polygon sides=6, thick,
					draw,
					minimum width=1.2cm,
					outer sep=0,
					},
					      transform shape]
					    \node (A) {};
					    \node (B) at (A.corner 1) {};
					    \node (C) at (B.corner 5) {};
					    \node (D) at (A.corner 5) {};
					    \node (E) at (B.corner 1) {};
				\end{scope}
			\end{tikzpicture}
		\end{center}
	\end{minipage} 
	\begin{minipage}[r]{.3\textwidth}
		\begin{center}
			\begin{tikzpicture}
				\draw[thick] (xyz polar cs:angle=30,radius=0.77) circle (0.666cm);
				\draw[thick] (xyz polar cs:angle=150,radius=0.77) circle (0.666cm);
				\draw[thick] (xyz polar cs:angle=270,radius=0.77) circle (0.666cm);
				\draw[thick] (0,0) circle (0.10333cm);
				\draw[thick] (xyz polar cs:angle=30,radius=0.77) circle (0.666cm);
				\draw[thick] (xyz polar cs:angle=150,radius=0.77) circle (0.666cm);
				\draw[thick] (xyz polar cs:angle=90,radius=1.11467) circle (0.322cm);
				\draw[thick] (xyz polar cs:angle=210,radius=1.11467) circle (0.322cm);
				\draw[thick] (xyz polar cs:angle=330,radius=1.11467) circle (0.322cm);
			\end{tikzpicture}
		\end{center}
	\end{minipage} 
	\caption{A hexagonal partition and a decomposition into circles. The centers of the hexagons generate a hexagonal lattice.}\label{fig_partition}
\end{figure}
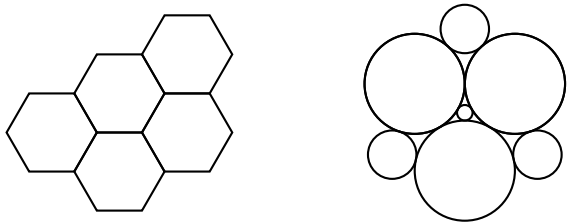

It is not surprising to expect the hexagonal packing to be effective, which arises traditionally whenever one wishes to partition a domain into circle-like domains of equal measure. The most famous result in that direction is certainly Hales' honeycomb theorem \cite{hales} stating that the hexagonal partition minimizes the average perimeter among all partitions into cells of equal measure (see Figure \ref{fig_partition}). There has also been a series of recent problems in mathematical physics (see e.g.\ Caffarelli \& Lin \cite{caff} and the survey of Helffer \& Hoffman-Ostenhof \cite{helffer}) and the calculus of variations \cite{steinerberger} related to the intuitive notion that 'the most circle-like partition of the plane is given by the hexagonal packing'.

We believe that it might be possible to establish connections between partitioning results and Gabor frame bounds with Gaussian window. This approach could provide a uniform lower bound $ B/A \geq 1+c$ for some universal $c>0$ over all lattices of a fixed redundancy and we consider this to be an interesting direction for further research (see also \cite{AbrDoe12}). Furthermore, we conjecture that for fixed redundancy and Gaussian window, the hexagonal lattice simultaneously maximizes the lower frame bound and minimizes the upper frame bound. It seems natural to assume that the hexagonal lattice might be optimal for a rather large family of window functions.

\section{Preliminary reductions}\label{sec_Frame_Bounds}

\subsection{Explicit formulae.} From this point on, we will only explore the reduced frame set and whenever we speak of the frame set we mean the reduced frame set. In this section we will present the explicit formulas of Janssen \cite{Jan96} for the upper and lower frame bound of Gabor frames with the standard Gaussian window on separable lattices. Let $g_0(t) = 2^{1/4} e^{-\pi t^2}$ be the standard Gaussian and $\L_{\alpha,\beta} = \alpha \Z \times \beta \Z$ with $(\alpha,\beta) \in \F(g)$ and $(\alpha \beta)^{-1} = n \in \N$. The frame bounds depend on the lattice parameters $\alpha, \beta$ and $n$. According to \cite{Jan96} the upper and lower frame bound are the essential supremum and infimum of the function
$$
	F(x,\omega;\alpha,\beta) = \frac{1}{\alpha \beta} \sum_{k \in \Z} \sum_{l \in \Z} (-1)^{\frac{k l}{\alpha \beta}} e^{-\frac{\pi}{2}\left( \frac{k^2}{\beta^2} + \frac{l^2}{\alpha^2} \right)} e^{2 \pi i kx} e^{2 \pi i l \omega}.
$$
By periodicity it suffices to consider $F$ on the unit square, $(x,\omega) \in [0,1] \times [0,1]$. In the case of even integer redundancy $(\alpha \beta)^{-1} \in 2 \N$, the alternating sign vanishes and we have
\begin{align*}
	F(x,\omega;\alpha,\beta) &= \frac{1}{\alpha \beta}\sum_{k \in \Z} \sum_{l \in \Z} e^{-\frac{\pi}{2}\left( \frac{k^2}{\beta^2} + \frac{l^2}{\alpha^2} \right)} e^{2 \pi i kx} e^{2 \pi i l \omega} \\
	&=    \frac{1}{\alpha \beta} \left(  \sum_{k \in \Z} e^{-\frac{\pi}{2} \frac{k^2}{\beta^2} } e^{2 \pi i kx}  \right)
	\left(  \sum_{l \in \Z} e^{-\frac{\pi}{2}  \frac{l^2}{\alpha^2}} e^{2 \pi i l \omega} \right)
\end{align*}
We rewrite this double sum using Jacobi's theta functions \cite{SteSha03}. Recall that for $z \in \mathbb{C}$ and $\tau \in \mathbb{H} = \lbrace \tau \in \mathbb{C} |~ \Im(\tau) > 0 \rbrace$ Jacobi's theta function is defined as
	$$
		\Theta(z,\tau) = \sum_{k=-\infty}^{\infty} e^{\pi i k^2 \tau} e^{2 \pi i k z}.
	$$
We can now rewrite the function arising in the case of even redundancy as
$$
	F(x,\omega;\alpha,\beta) = \frac{1}{\alpha \beta} \Theta\left(\omega,\frac{i}{2 \alpha^2}\right) \Theta\left(x,\frac{i}{2 \beta^2}\right).
$$
As Janssen already stated in \cite{Jan96} we find that $F$ takes its supremum for $(x, \omega) = (0,0)$ and its infimum for $(x, \omega) = \left(\frac{1}{2}, \frac{1}{2}\right)$. 

\subsection{Even redundancy} Fixing $(\alpha \beta)^{-1} = n$, with $n$ being even, we can write the lower and upper frame bound respectively as
$$
	A(\beta)  = n \, \Theta\left(\frac{1}{2}, \frac{i n^2 \beta^2}{2}\right) \Theta\left(\frac{1}{2},\frac{i}{2 \beta^2}\right) \quad \mbox{and} \quad
	B(\beta)  = n \, \Theta\left(0, \frac{i n^2 \beta^2}{2}\right) \Theta\left(0,\frac{i}{2 \beta^2}\right).
$$
This reduces the problem of optimizing $A$ and $B$ to optimization with respect to the lattice parameter $\beta$. We recall Jacobi's theta functions for the special arguments $z = 0$ and $z = 1/2$ and purely imaginary $\tau = i s$, $s \in \R_+$ to further simplify the statement of the problem
$$
	\theta_3(s) := \Theta(0, i s) = \sum_{k=-\infty}^{\infty} e^{-\pi k^2 s} \quad \mbox{and \quad}
	\theta_4(s) := \Theta(1/2, i s) = \sum_{k=-\infty}^{\infty} (-1)^k e^{-\pi k^2 s}.
$$
As in Janssen's work \cite{Jan96}, we can rewrite the frame bounds as
\begin{equation*}
	A(\beta) = n \, \theta_4\left(\frac{n^2 \beta^2}{2}\right) \theta_4\left(\frac{1}{2 \beta^2}\right) \text{ and } \,
	B(\beta) = n \, \theta_3\left(\frac{n^2 \beta^2}{2}\right) \theta_3\left(\frac{1}{2 \beta^2}\right).
\end{equation*}
The substitution $\beta = \sqrt{s/n}$ shows that the statements about the maximality of the lower frame bound and minimality of the upper frame bound follow from knowing that for $r, s > 0$,
$$
	\theta_4(rs) \theta_4\left(\frac{r}{s}\right) \leq \theta_4(r)^2 \quad \mbox{and} \quad
	\theta_3(rs) \theta_3\left(\frac{r}{s}\right) \geq \theta_3(r)^2 \qquad \mbox{with equality only for}~s=1.
$$

\subsection{Odd redundancy} The case of odd redundancy introduces an additional alternating sign. As announced above, we will use 
$$ \theta_{o}(s) = \sum_{k \in \mathbb{Z}}{e^{-\pi(2k+1)^2 s}}$$
to simplify representation. Using this function, we can rewrite the representation obtained by Janssen \cite{Jan96} in the case of $(\alpha \beta)^{-1} = n$ being odd as
\begin{align*}
	A(\beta) & = n \left(\theta_4\left(\frac{n^2 \beta^2}{2}\right) \theta_4\left(\frac{1}{2 \beta^2}\right) - 2 \theta_{o}\left(\frac{n^2 \beta^2}{2}\right) \theta_{o}\left(\frac{1}{2 \beta^2}\right)\right)\\
	B(\beta) & = n \left(\theta_3\left(\frac{n^2 \beta^2}{2}\right) \theta_3\left(\frac{1}{2 \beta^2}\right) - 2 \theta_{o}\left(\frac{n^2 \beta^2}{2}\right) \theta_{o}\left(\frac{1}{2 \beta^2}\right)\right).
\end{align*}

\subsection{A Useful Lemma.} A straight-forward analysis seems difficult because these functions have extremely small derivatives around $s=1$ (see Figure \ref{fig_theta3}). We proceed by exploiting the algebraic form of the term first and proving a slightly stronger statement.
\begin{lemma}\label{lem_useful_lemma}
	Let $F_r(s) = f(rs)f(r/s)$ with $f: \R_{+} \rightarrow \R_{+}$ differentiable and $r \in \R_+$ fixed. If
	$$
		s \frac{f'(s)}{f(s)} \qquad \mbox{is strictly increasing (decreasing) for}~s > 0,
	$$
	then $F_r(s)$ has its global minimum (maximum) and only critical point at $s=1$. 
\end{lemma}
\begin{proof}
	 The algebraic structure implies that $F_r(s) = F_r(1/s)$ and therefore there exists either a local minimum or a local maximum. Any critical point of $F_r$ satisfies
	$$
		 0 = \frac{d}{ds} F_r(s) = r f'(rs)f(r/s) - \frac{r}{s^2}f(rs)f'(r/s),
	$$
	which is equivalent to
	$$
		rs \frac{f'(rs)}{f(rs)} = \frac{r}{s} \frac{f'(r/s)}{f(r/s)}.
	$$
	The monotonicity assumption implies $rs = r/s$ and thus $s=1$ is the only solution and that the extremum is global.
\end{proof}

\begin{figure}[h!] 
	\centering
	\begin{tikzpicture}[scale=0.8]
			\begin{axis}[
			ymin=1, ymax=1.001, ytick={1,1.0001,1.0002,1.0003,1.0004,1.0005,1.0006,1.0007,1.0008,1.0009}, yticklabels={1,1.0001,1.0002,1.0003,1.0004,1.0005,1.0006,1.0007,1.0008,1.0009}]
			\addplot[samples=200,domain=0.4:2.5, ultra thick]{0.000002 +  ( 1+2*exp(-pi*6*x)))*(1+2*exp(-pi*(6/x)))       };
		\end{axis}
	\end{tikzpicture}
	\captionsetup{width=0.95\textwidth}
	\caption{The function $\theta_3(6s)\theta_3(6/s)$ near $s=1$.}\label{fig_theta3}
\end{figure}
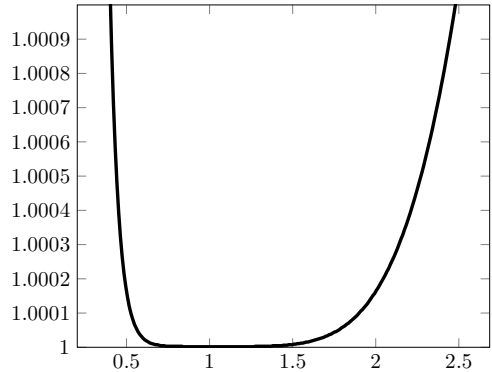
 
Note that 
$$
	\frac{d}{ds} \left(s \frac{f'(s)}{f(s)} \right)= \frac{f'(s)}{f(s)} + s\frac{f''(s)f(s) - f'(s)^2}{f(s)^2}
$$
and setting that expression to be positive or negative and rearranging gives precisely the expressions in Theorem \ref{log_convex} and Theorem \ref{log_concav}. Finally, we note a general log-convexity statement about functions of the type $\sum_{k=1}^{\infty}{a_k e^{- b_k s}}$, which implies the log-convexity of $\theta_3$.
\begin{lemma}\label{lem_theta3_log}
	Let $a_k, b_k \geq 0$ be sequences of positive real numbers such that $\sum_{k}{a_k} < \infty$. Then
	$$
		f(s) = \sum_{k=1}^{\infty}{a_k e^{- b_k s}} \qquad \mbox{satisfies} \quad \left(\log{(f(s))}\right)'' \geq 0.
	$$
\end{lemma}
\begin{proof} 
	We prove the statement in the form $ f''(s)f(s) - f'(s)^2 \geq 0$. Let
	$$
		f(s) = \sum_{k=1}^{\infty}{a_k e^{- b_k s}}.
	$$
	Then, by direct computation, our statement can be written as
	$$
	f''(s)f(s) - f'(s)^2 = \left( \sum_{k=1}^{\infty}{b_k^2 a_k e^{- b_k s}}\right) \left( \sum_{k=1}^{\infty}{a_k e^{- b_k s}}\right) -  \left(\sum_{k=1}^{\infty}{b_k a_k e^{- b_k s}}\right)^2 \geq 0.
	$$
	Using the fact that
	\begin{equation*}
		\sum_{k = 1}^\infty a_k b_k e^{-b_k s} = \sum_{k = 1}^\infty \left( \sqrt{a_k} b_k e^{-b_k s/2} \right) \left( \sqrt{a_k} e^{-b_k s/2} \right),
	\end{equation*}
	the proof follows by applying the Cauchy-Schwarz inequality.
\end{proof}

\section{Proof of Theorem \ref{log_convex}}
We will show that for $r >0$ and
$$
	\theta_3(s) = \sum_{k \in \Z} e^{-\pi k^2 s} = 1 + 2\sum_{k \geq 1} e^{-\pi k^2 s}
$$
we have 
\begin{equation*}\label{eq_proof_theta3}
	\theta_3(rs)\theta_3\left(\frac{r}{s}\right) \geq \theta_3(r)^2 \qquad \mbox{with equality only for $s=1$ (first part of Theorem \ref{alternativ}).}
\end{equation*}
This will imply the statement about the upper frame bound for even redundancy in Theorem \ref{main} and follows from the algebraic Lemma \ref{lem_useful_lemma} if we can show the following (see Figure \ref{fig_theta3_identity}.
\begin{theorem}
	$$
		s \frac{\theta_3'(s)}{\theta_3(s)} \text{ is strictly increasing on } \R_+.
	$$
\end{theorem}
This statement follows immediately from the upcoming two facts which we will now prove.
\begin{itemize}
	\item \textbf{Fact 1.} The function  $ s \theta_3'(s)/\theta_3(s)$ is strictly increasing for $s \geq 1$.
	\item \textbf{Fact 2.} For all $s > 0$, we have 
	$$
		s \frac{\theta_3'(s)}{\theta_3(s)} + \frac{1}{s}\frac{\theta_3'\left(\frac{1}{s}\right)}{\theta_3\left(\frac{1}{s}\right)} = -\frac{1}{2}.
	$$
\end{itemize}

\begin{figure}[ht!]
	\centering
	\begin{tikzpicture}[scale=0.8]
		\begin{axis}[
			ytick={-0.5,-0.25,0}, yticklabels={-0.5,-0.25,0}]
			\addplot[samples=200,domain=0.1:3, ultra thick]{    x*(-50*pi*exp(-25*pi*x) -pi*32*exp(-16*pi*x)  -pi*18*exp(-9*pi*x) -pi*8* exp(-4*pi*x) - pi*2*exp(-pi*x))/(1+2*exp(-25*pi*x) +2*exp(-16*pi*x)+2*exp(-9*pi*x)+2*exp(-4*pi*x)  +2*exp(-pi*x))      };
		\end{axis}
	\end{tikzpicture}
	\captionsetup{width=0.95\textwidth}
	\caption{The function $ s \theta_3'(s)/\theta_3(s)$.}\label{fig_theta3_identity}
\end{figure}
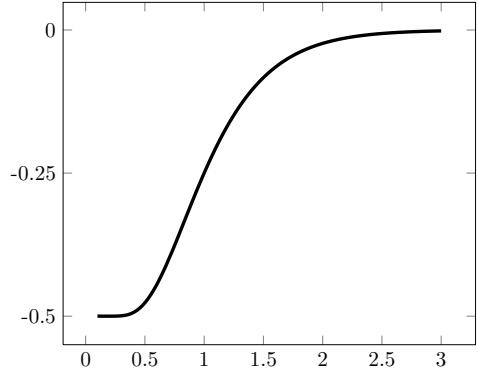

\begin{lem}[Fact 1]
	The function  $ s \theta_3'(s)/\theta_3(s)$ is strictly increasing for $s \geq 1$.
\end{lem}

\begin{proof}
	We will now show that the derivative is positive. This is equivalent to showing that
	$$
		s \theta''_3(s)\theta_3(s) + \theta'_3(s)\theta_3(s) - s \theta'_3(s)^2 > 0,
	$$
	where the first term is positive and the last two terms are negative. Since, $\theta_3(s) \geq 1$ the statement above is implied by the stronger inequality
	$$
		s \theta''_3(s) + \theta'_3(s)\theta_3(s) - s \theta'_3(s)^2 > 0,
	$$
	which can be equivalently written as
	$$
		\underbrace{s \left(2\sum_{k \geq 1}{\pi^2 k^4 e^{- \pi k^2 s}} \right)}_{(I)} - \underbrace{\left(2\sum_{k \geq 1}{\pi k^2 e^{- \pi k^2 s}}\right)}_{(II)} \underbrace{\left(1+2\sum_{k \geq 1}{e^{- \pi k^2 s}}\right)}_{\theta_3(s)} - s\underbrace{\left( 2 \sum_{k \geq 1}\pi k^2 {e^{- \pi k^2 s}}\right)^2}_{(III)} > 0.
	$$
	We will now establish this inequality using term-by-term estimates. The first term $(I)$ is bounded from below by
	$$
		(I) = s \left(2\sum_{k \geq 1}{\pi^2 k^4 e^{- \pi k^2 s}} \right) \geq 2s \pi^2 e^{- \pi s} .
	$$
	In order to control $(II)$ and $(III)$ we bound expressions using geometric series. More precisely, we use
	$$
		\sum_{k \geq m} q^k  = \frac{ q^m}{1-q} \quad \mbox{and} \quad  \sum_{k \geq m} k q^k = q^m \frac{m-(m-1)q}{(1-q)^2}.
	$$
	This gives
	\begin{align*}
		(II) &= 2\sum_{k \geq 1}{\pi k^2 e^{- \pi k^2 s}} \leq 2\pi e^{-\pi s}  + 2\pi  \sum_{k \geq 4}{k e^{- \pi k s}} = 2\pi \left(e^{-\pi s} + e^{-4 \pi s}\frac{(4 - 3 e^{-\pi s})}{(1-e^{-\pi s})^2} \right) \\
		\theta_3(s) &= 1+2\sum_{k \geq 1}{e^{- \pi k^2 s}} \leq 1 + 2\sum_{k \geq 1}{e^{- \pi k s}} = 1 + 2 \frac{e^{-\pi s}}{1-e^{-\pi s}} \\
		(III) &\leq 4\pi^2 \left(e^{-\pi s} + e^{-4 \pi s}\frac{(4 - 3 e^{-\pi s})}{(1-e^{-\pi s})^2} \right)^2.
	\end{align*}
	These estimates imply the statement for $s \geq 0.7$ (and, in particular, for $s \geq 1$).
\end{proof}

\begin{lem}[Fact 2] We have, for all $ s > 0$, 
	$$
		s \frac{\theta'_3(s)}{\theta_3(s)} + \frac{1}{s} \frac{\theta'_3(1/s)}{\theta_3(1/s)} = -\frac{1}{2}.
	$$
\end{lem}
\begin{proof}
	We use the Jacobi identity
	$$
		\theta_3\left(\frac{1}{s}\right) =  \sqrt{s} \theta_3(s)
	$$
	which is easily derived by using Poisson summation (see e.g.\ \cite{SteSha03}). Differentiating the identity on both sides gives
	$$
		-\frac{1}{s^2}\theta_3'\left(\frac{1}{s}\right) = \frac{1}{2 \sqrt{s} }\theta_3(s) + \sqrt{s} \theta_3'(s).
	$$
	Multiplying with $-s$ yields
	$$
		\frac{1}{s} \theta_3'\left(\frac{1}{s}\right) = - \frac{\sqrt{s}}{2} \theta_3(s) - s^{3/2} \theta_3'(s),
	$$
	which we now divide by $\sqrt{s} \theta_3(s)$ (which equals $\theta_3(1/s)$ because of the Jacobi identity)
	$$
	-\frac{1}{2} - s \frac{\theta_3'(s)}{\theta_3(s)} =  \frac{1}{s} \frac{ \theta_3'\left(\frac{1}{s}\right) }{\sqrt{s} \theta_3\left(s\right) } = \frac{1}{s} \frac{ \theta_3'\left(\frac{1}{s}\right) }{ \theta_3\left(\frac{1}{s}\right) }.
	$$
\end{proof}
Theorem \ref{log_convex} follows, since
$$
	0 < \frac{d}{ds}\left(s \frac{\theta'_3(s)}{\theta_3(s)} \right) = \frac{\theta'_3(s)}{\theta_3(s)} + s \frac{\theta''_3(s) \theta_3(s)-\theta'_3(s)^2}{\theta_3(s)^2}
$$
which is equivalent to
$$
	\theta_3''(s)\theta_3(s) - \theta_3'(s)^2 > -\frac{\theta_3'(s) \theta_3(s)}{s}.
$$
We finally show that the right-hand side of the last inequality is positive. Since,
$$
	\theta_3(s) = 1 + 2 \sum_{k \in \N} e^{-\pi k^2 s} > 1
$$
and
\begin{equation*}
	\theta'_3(s) = - 2 \pi \sum_{k \in \N} k^2 e^{-\pi k^2 s} < 0,
\end{equation*}
we see that
$$
	\frac{\theta_3'(s) \theta_3(s)}{s} < 0
$$
and Theorem \ref{log_convex} is proved.

\section{Proof of Theorem \ref{log_concav}}
We will show that for $r \in \R_+$ fixed we have
\begin{equation*}
	\theta_4(rs)\theta_4\left(\frac{r}{s}\right) \leq \theta_4(r)^2
\end{equation*}
with equality only for $s=1$ (second part of Theorem \ref{alternativ}). This will imply the statement about the lower frame bound for even redundancy in Theorem \ref{main} and follows again from the algebraic Lemma \ref{lem_useful_lemma}. The proof uses the Jacobi triple product representation \cite{SteSha03}. The proof has some curious algebraic simplifications and we know of no other way to establish the result in its full generality.
\begin{theorem}\label{thm_theta4_monotone}
	$$
			s \frac{\theta'_4(s)}{\theta_4(s)} \qquad \mbox{is strictly decreasing on}~\R_+.
	$$
\end{theorem}
\begin{proof}
	We use the Jacobi triple product representation for $s \in \R_+$
	$$
		\theta_4(s) = \prod_{k \geq 1}  \left(1-e^{-2 k \pi s}\right) \left(1- e^{-(2 k-1) \pi s}\right)^2  = \prod_{k=1}^\infty \theta_{4,k}(s).
	$$
	Using the product rule we will show that for every $k \in \mathbb{N}$ and $s>0$ 
	$$
	 \frac{d}{ds}\left(s \frac{\theta'_{4,k}(s)}{\theta_{4,k}(s)}\right) < 0
	 $$
	 which then immediately implies
	$$
		\frac{d}{ds} \left(s \frac{\theta'_4(s)}{\theta_4(s)}\right) =  \frac{d}{ds} \left(s \frac{\left( \prod_{k=1}^\infty \theta_{4,k}(s) \right)' }{ \prod_{k=1}^\infty \theta_{4,k}(s)}\right)    = \frac{d}{ds} \left(s \sum_{k \geq 1}  \frac{\theta'_{4,k}(s)}{\theta_{4,k}(s)}\right) = \sum_{k \geq 1} \frac{d}{ds}\left(s \frac{\theta'_{4,k}(s)}{\theta_{4,k}(s)}\right) < 0.
	$$
	A simple computation yields
		\begin{align*}
			s \frac{\theta'_{4,k}(s)}{\theta_{4,k}(s)} & = s \left(\frac{2k\pi e^{-2k\pi s}}{1-e^{-2k\pi s}} + 2\frac{(2k-1)\pi e^{-(2k-1)\pi s}}{1-e^{-(2k-1)\pi s}}\right)\\
			& = \frac{2k\pi s}{e^{2k\pi s}-1} + 2\frac{(2k-1)\pi s}{e^{(2k-1)\pi s}-1}.
		\end{align*}
	
	Both terms in the last sum are of the form $ms/(e^{ms}-1)$ for some $m \in \R_+$. Note that
	$$
		\frac{d}{ds}\left(\frac{ms}{e^{ms}-1}\right) = m \frac{e^{ms}-(1+ms\, e^{ms})}{\left(e^{ms}-1\right)^2} < 0,
	$$
	which is quickly checked using the elementary inequality
	$$
		e^{\lambda} < 1 + \lambda \, e^{\lambda}  \quad \Leftrightarrow \quad 1 - \lambda < e^{-\lambda} \qquad \mbox{for all}~\lambda > 0.
	$$
	Therefore, the statement follows since all terms involved are negative.
\end{proof}
The part concerning the lower frame bound in Theorem \ref{main} and the refined log-concavity statement in Theorem \ref{alternativ} follow immediately. Theorem \ref{log_concav} follows, since
$$
	0 > \frac{d}{ds}\left(s \frac{\theta'_4(s)}{\theta_4(s)} \right) = \frac{\theta'_4(s)}{\theta_4(s)} + s \frac{\theta''_4(s) \theta_4(s)-\theta'_4(s)^2}{\theta_4(s)^2}
$$
which is equivalent to
$$
	\theta_4''(s)\theta_4(s) - \theta_4'(s)^2 < -\frac{\theta_4'(s) \theta_4(s)}{s}.
$$
We finally show that the right-hand side of the last inequality is negative. Since,
$$
	\theta_4(s) = 1 + 2 \sum_{k \in \N} (-1)^k e^{-\pi k^2 s} \qquad \mbox{we have} \qquad 	\theta'_4(s) = - 2 \pi \sum_{k \in \N} (-1)^k k^2 e^{-\pi k^2 s}.
$$
We see that $\lim_{s \to \infty} \theta_4(s) = 1$ and that $\lim_{s \to \infty} s \theta'_4(s) = 0$ therefore
$$
	\lim_{s \to \infty} s \, \frac{\theta'_4(s)}{\theta_4(s)} = 0.
$$
We just proved that this term is strictly decreasing, this implies that
$$
	s \frac{\theta'_4(s)}{\theta_4(s)} > 0 \qquad \mbox{for all}~s > 0 \qquad \mbox{and thus} \qquad \frac{\theta'_4(s) \theta_4(s)}{s} > 0
$$
which completes Theorem \ref{log_concav}.
\section{Proof of Theorem \ref{oddupp}}
We want to show that for $r \in \R_+$ fixed
$$
	\theta_3(rs) \theta_3\left(\frac{r}{s}\right) - 2 \theta_{o}(r s) \theta_{o}\left(\frac{r}{s}\right) \qquad \mbox{is minimal for}~s=1.
$$
This implies the statement about the upper frame bound in Theorem \ref{main} for odd redundancy. From Theorem \ref{alternativ} we know that $\theta_3(rs) \theta_3\left(\frac{r}{s}\right)$ is minimal for $s=1$.
\begin{theorem}\label{omax}
	Let $r \in \R_+$. The quantity
	$$
		\theta_{o}(r s) \theta_{o}\left(\frac{r}{s}\right) \qquad \mbox{is maximal for}~s=1.
	$$
\end{theorem}
\begin{proof}
	Using Poisson's formula the result follows from Theorem \ref{alternativ}. We have
	\begin{equation*}
		\theta_o(rs) = \sum_{k \in \Z} e^{-\pi (2k-1)^2 r s} = \sum_{k \in \Z} e^{-4 \pi (k-1/2)^2 r s} = \frac{1}{2 \sqrt{r s}} \sum_{k \in \Z} e^{-\pi i k} e^{-\frac{\pi k^2}{4 r s}} = \frac{1}{2\sqrt{r s}} \theta_4\left(\frac{1}{4 r s}\right).
	\end{equation*}
	In the same manner we derive
	\begin{equation*}
		\theta_o(r/s) = \frac{\sqrt{s}}{2\sqrt{r}} \theta_4\left(\frac{s}{4 r}\right).
	\end{equation*}
	Therefore, we have
	\begin{equation*}
		\theta_o(rs)\theta_o(r/s) = \frac{1}{4r} \, \theta_4\left(\frac{1}{4 r s}\right) \theta_4\left(\frac{s}{4 r}\right)
	\end{equation*}
	and the statement follows from the results about the $\theta_4$ function.
\end{proof}

There exists another argument that suffices to prove a slightly weaker result (needing $r \geq 1/2$) that would still be sufficient. We present it because the proof is slightly more flexible and could also be used to give an alternative proof of a weaker (but also sufficient) version of Theorem \ref{log_convex}. 
\begin{lemma} \label{lemming}
	$$
		s\frac{ \theta_{o}'(s)}{\theta_{o}(s)} \qquad \begin{cases} \mbox{is strictly decreasing on}~\left\{s:s \geq \frac{1}{4}\right\}\\
		\geq  \theta_{o}'(1)/\theta_{o}(1)\qquad \mbox{for all}~s \leq \frac{1}{4}.\end{cases}
	$$
\end{lemma}

	\begin{proof}[Proof of Theorem \ref{omax} assuming Lemma \ref{lemming}]
	We appeal again to the algebraic lemma: suppose that 
		$$	rs \frac{\theta_o'(rs)}{\theta_o(rs)} = \frac{r}{s} \frac{\theta_o'(r/s)}{\theta_o(r/s)}$$
	for some $r \geq 1/2$ and some $s \neq 1$. Since 
	$ (rs) (r/s) = r^2 \geq 1/4,$
	we know that at least one of the two terms $rs$ and $r/s$ is bigger than $1/2$.  If both terms are bigger than $1/4$, then the first part of Lemma \ref{lemming} implies
	that $rs = r/s$ and thus $s=1$. If one term (w.l.o.g. $rs$) is smaller than $1/4$, then $r/s = r^2/(rs) \geq (1/4)/(1/4) = 1$ and both parts of Lemma \ref{lemming} imply
	$$ rs \frac{\theta_o'(rs)}{\theta_o(rs)} \geq  \frac{\theta_o'(1)}{\theta_o(1)} >  \frac{r}{s} \frac{\theta_o'(r/s)}{\theta_o(r/s)}.$$
	\end{proof}

\begin{proof}[Proof of Lemma \ref{lemming}]
For $s$ bounded away from the origin, the quantity
	$$
	  s\frac{ \theta_{o}'(s)}{\theta_{o}(s)} = \frac{ - \sum_{k \in \mathbb{Z}}{ s \pi (2k+1)^2 e^{-\pi (2k+1)^2 s}}}{  \sum_{k \in \mathbb{Z}}{e^{-\pi (2k+1)^2 s}}}
	$$
	can be controlled because of the strong decay properties of all terms. Indeed, we have
	$$
		\left|s\frac{ \theta_{o}'(s)}{\theta_{o}(s)}  +  \pi s \right| \sim 9 \pi s e^{-8 \pi s}\qquad \mbox{as}~s \rightarrow \infty .
	$$
	Standard estimates allow to control the error on $\left\{s: s \geq 1/4\right\}$. The second half of the statement is slightly trickier because, many terms suddenly contribute non-trivially to the infinite sum. There is another algebraic simplification: the additional factor $s$ in front of $\theta_{o}'(s)/\theta_{o}(s)$ implies that both numerator and denominator can be interpreted as Riemann sums:
\begin{align*}
 \sum_{k \in \mathbb{Z}}{  \pi \left[ \sqrt{s} (2k+1)\right]^2e^{-\pi \left[ \sqrt{s}(2k+1) \right]^2}} &\sim \frac{1}{2\sqrt{s}} \int_{0}^{\infty}{\pi z^2 e^{-\pi z^2}dz} = \frac{1}{8\sqrt{s}}\\
  \sum_{k \in \mathbb{Z}}{e^{-\pi \left[ \sqrt{s}(2k+1) \right]^2}} &\sim \frac{1}{2\sqrt{s}} \int_{0}^{\infty}{e^{-\pi z^2}dz} = \frac{1}{4\sqrt{s}}.
\end{align*}
This implies, in particular, that
$$ \lim_{s \rightarrow 0^+}{ s\frac{ \theta_{o}'(s)}{\theta_{o}(s)} } = -\frac{1}{2}.$$
The error estimates that control the deviation from the Riemann sum and the integral are standard since $e^{-\pi z^2}$ is monotone and $\pi z^2 e^{-\pi z^2}$ is unimodal. 
\end{proof}

\section{Proof of Theorem \ref{oddlow}}
We want to show that for $r \geq 1$
$$
	\theta_4(rs) \theta_4\left(\frac{r}{s}\right) - 2 \theta_{o}(r s) \theta_{o}\left(\frac{r}{s}\right) \qquad \mbox{is maximal for}~s=1.
$$
This implies the statement about the lower frame bound in Theorem \ref{main} for odd redundancy.
\begin{proof}
This argument works directly by establishing bounds on both size and derivative of the function. Note that for $r > 0$ fixed we already know (Theorem \ref{alternativ} and Theorem \ref{omax}) that both $ \theta_4(rs) \theta_4(r/s)$ and $ \theta_{o}(r s) \theta_{o}(r/s)$ are maximal for $s=1$. In particular, there exists a critical point in $s=1$. Because of the definition as alternating sum, we immediately have
$$ \theta_4(x) = 1 - 2 e^{-\pi x} + 2 e^{- 4 \pi x}  + (- 2 e^{- 9 \pi x} + 2 e^{- 16 \pi x} ) + \dots \leq  1 - 2 e^{-\pi x} + 2 e^{- 4 \pi x}$$
and therefore
\begin{align*}
	\theta_4(rs) \theta_4(r/s) &\leq (1-2e^{-\pi r s}+ 2e^{-4 \pi r s})(1-2e^{-\pi r/s}+2e^{-4\pi r/s})\\
	\theta_{o}(rs) \theta_{o}(r/s) &\leq  \theta_{o}(r)^2 \sim 4 e^{- 2 \pi r}
\end{align*}
Note that for $r \geq 1$ and $1/3 \leq s \leq 3$ the expansion are fairly accurate (and could be made more precise by adding additional terms).
This chain of inequalities implies that the function cannot assume a global maximum for $s$ outside of the range $1/3 \leq s \leq 3$. We have approximately
$$ \theta_4(r)^2 \sim (1-2 e^{-\pi r})^2 \sim 1 - 4 e^{-\pi r}.$$
At the same time, we have for $s \leq 1/3$ that 
$$ \theta_4(rs) \leq \theta_4\left(\frac{r}{3}\right) \leq 1-2e^{-\pi \frac{r}{3}}+ 2e^{-4 \pi \frac{r}{3}}$$
and
$$ 2e^{-\pi \frac{r}{3}} -  2e^{-4 \pi \frac{r}{3}} \gg 4 e^{-\pi r} \geq 4 e^{-2 \pi r} \qquad \mbox{for}~r \geq 1.$$
The same reasoning works for $s \geq 3$ by using the same argument on the other term.
We can thus restrict ourselves to $1/3 \leq s \leq 3$ and show that the second derivative of the first term is many order of magnitudes bigger than the second derivative of the second term which will then imply the statement. Note that the proper first-order approximations are
\begin{align*}
	\theta_4( rs) \theta_4\left(\frac{r}{s}\right) &\sim 1 - 2 e^{-\pi r s}  - 2 e^{-\pi \frac{r}{s}} + 4 e^{-\pi \left(rs + \frac{r}{s}\right)} \\
	\theta_{o}( rs) \theta_{o}\left(\frac{r}{s}\right) &\sim 4 e^{-\pi \left(rs + \frac{r}{s}\right)}
\end{align*}
We see that the second function is merely the second-order correction of the first function and because we may assume $1/3 \leq s \leq 3$, easy local estimates suffice to establish the result.
We leave the details to the interested reader.
\end{proof}

\subsection*{Acknowledgment} The authors wish to thank Karlheinz Gr\"ochenig for many fruitful discussions on the topic. The authors wish to thank Thomas Strohmer for beneficial feedback and for pointing out a reference. The first author was supported by the Austrian Science Fund (FWF): [P26273-N25]. The second author was supported by a Yale Provost Travel Grant and INET Grant \#INO15-00038. This project was initiated at the Oberwolfach Workshop 1534 \textsc{Applied Harmonic Analysis and Sparse Approximation} and the authors are grateful to both the MFO and the Organizers for the enjoyable and productive week.

\bibliographystyle{plain}
\bibliography{diss}

\end{document}